\documentclass[english]{amsart}

\usepackage{esint}
\usepackage[svgnames]{xcolor} 
\usepackage{color}
\usepackage[colorlinks,citecolor=red,pagebackref,hypertexnames=false,breaklinks]{hyperref}
\usepackage{pgf,tikz}
\usepackage{pdfsync}

\usepackage{dsfont}
\usepackage{url}
\usepackage[utf8]{inputenc}
\usepackage[T1]{fontenc}
\usepackage{lmodern}
\usepackage{babel}
\usepackage{mathtools}  
\usepackage{amssymb}
\usepackage{lipsum}
\usepackage{mathrsfs}
\usepackage{color}
\usepackage[skip=2.2pt plus 1pt, indent=12pt]{parskip}
\usepackage{stmaryrd}
\usepackage{soul}
\usepackage{ulem}

\newtheorem{thm}{Theorem}[section]
\newtheorem{lem}[thm]{Lemma}
\newtheorem{cor}[thm]{Corollary}
\newtheorem{prop}[thm]{Proposition}

\newtheorem{rmk}{Remark}

\numberwithin{equation}{section}

\newcommand{\bel}{\begin{equation} \label}
\newcommand{\ee}{\end{equation}}
\def\beq{\begin{equation}}
\def\eeq{\end{equation}}
\newcommand{\bea}{\begin{eqnarray}}
\newcommand{\eea}{\end{eqnarray}}
\newcommand{\beas}{\begin{eqnarray*}}
\newcommand{\eeas}{\end{eqnarray*}}
\newcommand{\pd}{\partial}

\newcommand{\dd}{\mathrm{d}}

\newcommand{\ep}{\varepsilon}
\newcommand{\la}{\lambda}

\newcommand{\proj}{\mathbb{P}}
\newcommand{\ind}{\mathds{1}}
\newcommand{\B}{\mathbb{B}}
\newcommand{\R}{\mathbb{R}}
\newcommand{\C}{\mathbb{C}} 
\newcommand{\N}{\mathbb{N}}

\newcommand{\cB}{\mathcal{B}}

\newcommand{\cF}{\mathcal{F}}

\newcommand{\cN}{\mathcal{N}}

\newcommand{\cS}{\mathcal{S}}

\newcommand{\supp}{\mathrm{supp}\,}  

\allowdisplaybreaks

\def\phi {\varphi}
\def \la {{\lambda}}

\providecommand{\abs}[1]{\left\lvert#1\right\rvert}
\providecommand{\norm}[1]{\left\lVert#1\right\rVert}

\renewcommand{\leq}{\leqslant}
\renewcommand{\geq}{\geqslant}

\providecommand{\abs}[1]{\left\lvert#1\right\rvert}
\providecommand{\norm}[1]{\left\lVert#1\right\rVert}

\title[Determining an Iwatsuka Hamiltonian]{Determining an Iwatsuka Hamiltonian by knowledge of its first band function}

\author{Mourad Choulli}
\address{Universit\'e de Lorraine, France}
\email{mourad.choulli@univ-lorraine.fr}

\thanks{The first and third authors were partially supported by the Agence Nationale de la Recherche (ANR) under grant ANR-17-CE40-0029 (projet MultiOnde) while doing this work.}

\author{Nour Kerraoui}
\address{Aix-Marseille Univ, CNRS, I2M, Marseille, France}
\email{nour-el-houda.kerraoui@etu.univ-amu.fr}

\author{\'Eric Soccorsi}
\address{Aix Marseille Univ, Universit\'e de Toulon, CNRS, CPT, Marseille, France}
\email{eric.soccorsi@univ-amu.fr}

\date{}
 

\begin{document}

\begin{abstract}
We investigate the inverse problem of retrieving the magnetic potential of an Iwatsuka Hamiltonian through knowledge of its first band function. We prove that two magnetic potentials sharing the same first band fonction not just between them, but between all fields within the family linearly interpolating between them and accumulating at one end point, coincide. 

\end{abstract}

\maketitle

\tableofcontents

\section{Introduction: Iwatsuka Hamiltonians, edge currents and magnetic inverse problems}

Quantum Hall Hamiltonians describe the motion of a charged particle constrained to a bounded or unbounded subdomain of the plane, subject to a constant transverse magnetic 
field with strength $b \ge 0$.
Confined quantum Hall systems, such as motion in a half-plane or a strip are particularly interesting as a current flowing along an edge is created. Confinement may be obtained by Dirichlet boundary conditions (hard edge) or an electrostatic potential barrier (soft edge), but in any case the edges of the confinement induce edge currents. 
These edge currents are carried by states with energy localized between any two Landau levels $(2n-1)b$, $n \in \N :=\{ 1, 2, \ldots \}$, see e.g.
\cite{CHS,FGW, HS1, HS2}.

In the present article we are interested in edge currents created by purely magnetic barriers. Namely, we consider a two-dimensional Schr\"odinger operator with a non-constant magnetic field $b(x,y)=b(x)$, $(x,y) \in \R^2$, depending only on $x$.  When the real-valued function $b$ is bounded and has different limits as $x$ goes to $\pm \infty$, it was shown in \cite{I} by
Iwatsuka  that the spectrum is absolutely continuous. Later on, the transport properties of these so-called Iwatsuka Hamiltonians were investigated in \cite{RP} by physicists Reijniers and Peeters. When $b(x)$ assumes constant value $b_\pm$ for $\pm x > 0$, $0<b_- < b_+<\infty$, they argued that this discontinuity in the magnetic field at $x=0$ creates an effective edge and that currents flow along the edge. This is the magnetic analog of the barriers created by Dirichlet boundary conditions along $x = 0$ or a confining electrostatic potential filling the half-space $x < 0$ described in the paragraph above. 

Besides, 
when $0 < b_- \le b(x) \le b_+$ for $x \in (-\epsilon,\epsilon)$, where $ \epsilon \in [0,b_-^{-1/2})$, and
$b(x) = b_\pm$ for $\pm x > \epsilon$, the existence of currents flowing in the $y$-direction was rigorously established in \cite{HS3}. These currents are carried by states with energy concentration in the energy bands of the Iwatsuka Hamiltonian,
and they are well-localized in $x$ to a region of size
$b_-^{-1/2}$, centered at $x=0$. 
The study of the case of a jump in the magnetic field at $x = 0$ corresponding to $\epsilon = 0$ in \cite{HS3}, was extended to $b_+=-b_-=b>0$ 
in \cite{DHS}. It shows that the magnetic field creates an effective barrier near $x = 0$ that causes edge currents
to flow along it consistent with the heuristic approach conducted in \cite{RP}. Moreover, partially motivated by \cite{RP}, Dombrowski, Germinet and Raikov studied the edge 
conductance for generalized Iwatsuka 
models in \cite{DGR}.

In this article we examine the inverse problem of determining the vector potential $A(x,y)=(0,a(x))$, where $a(x)=\int_0^x b(s) \dd s$, $x \in \R$, of the Iwatsuka Hamiltonian 
defined on the dense domain $C_0^\infty(\R^2) \subset L^2(\R^2)$ by $(-i \nabla -A)^2$, from knowledge of the edge currents carried by quantum 
states with energy localized in the first band $(b_-,b_+)$. More precisely, these edge currents are generated upon triggering the dynamic quantum 
system governed by the Iwatsuka Hamiltonian, by a suitable set of initial states with energy concentration within $(b_-,b_+)$. 
This problem is closely related to the inverse spectral problem of determining the Iwatsuka Hamiltonian from knowledge of its first band function, that will be examined separately.

Inverse coefficient problems for the magnetic Schr\"odinger operator have attracted a great deal of attention over the last years. For instance, in \cite{CS},
using the Bukhgeim-Klibanov method, see \cite{BK}, the time-independent divergence-free magnetic potential was Lipschitz stably retrieved from a finite number of partial boundary observations (over the entire course of time) of the solution. The authors proceed by suitably changing the initial state of the corresponding dynamic Schr\"odinger equation and then measuring the solution on a sub-boundary fulfilling the geometric optics condition for the observability derived by Bardos, Lebeau and Rauch in \cite{BLR}. The case of non-zero divergence magnetic vectors was treated with the same approach in \cite{HKSY}.

In \cite{DKSU,S} the magnetic field of the Schr\"odinger operator was identified by the Dirichlet-to-Neumann map. These results are based on a different approach using geometric optics solutions. The stability issue in the same problem was treated in \cite{BC} but for dynamic magnetic Schr\"odinger 
equation in a bounded domain, and in \cite{B} for the same equation on a Riemannian manifold. As for the determination of the magnetic vector potential of the Schr\"odinger operator by spectral data, it was established in \cite{Ki}. All the above mentioned inverse results were derived for magnetic Schr\"odinger systems in a bounded domain, and we refer the reader to \cite{BKS} for the study of the inverse problem of retrieving the magnetic field of the Schr\"odinger 
equation in an unbounded cylindrical domain. 

It is worth noticing that the analysis of the inverse problem under study in this manuscript is fundamentally different from the ones used for solving the magnetic inverse problems of \cite{B, BC, BK, BKS, CS, DKSU, Ki, S}. This is mostly due to the fact that our data are not naturally related to the Neumann boundary data or the spectral data used by \cite{B, BC, BK, BKS, CS, DKSU, Ki, S}. 
Nevertheless, due to its translational invariance in the $y$ direction, the Iwatsuka Hamiltonian admits a fiber decomposition. The fibers are Sturm-Liouville operators with a zero-th order perturbation expressed in terms of the unknown function $a$. But despite of this, and since the fibers are defined on the real line here, it is still unclear whether the analysis conducted in \cite[Chapter 2, Section 3]{M} on the half-line could be adapted to the inverse problem under investigation in this manuscript. As for inverse spectral problems for one-dimensional Schr\"odinger operators defined either on the real-line or on the half-line, we refer the reader to \cite{GS} and the references therein, where the continuous and bounded from below real-valued electric potential was identified by 
the Krein spectral shift function.

\section{Definitions and results}
\label{sec-def}

\subsection{Iwatsuka Hamiltonians}
\label{sec-settings}

Let $b \in L^\infty(\R)$ be a non-decreasing function satisfying
\bel{a1}
\lim_{x \to \pm \infty} b(x)=b_\pm,
\ee
where $b_\pm$ are two positive real numbers such that
\bel{a2}
0 < b_-<b_+<3b_-.
\ee
Notice for further use that \eqref{a1} yields that
\bel{a2b}
b_- \le b(x) \le b_+,\ x \in \R.
\ee
Next, we put
\bel{a3}
a(x):=\int_0^x b(s) \dd s,\ x \in \R,
\ee
and we introduce the magnetic potential $A(x)=(A_1(x), A_2(x))$ with $A_1(x)=0$ and $A_2(x)=a(x)$. In the present article we consider the two-dimensional
magnetic Hamiltonian $(-i \nabla -A)^2$, defined on $C_0^\infty(\R^2)$ by
\bel{a4}
H :=-\partial_x^2 + (-i \partial_y - a)^2.
\ee
Since $A \in L_{\mathrm loc}^4(\R^2,\R^2)$ and $b=\partial_x A_2- \partial_y A_1 \in L_{\mathrm loc}^2(\R^2,\R)$, the operator $H$ is essentially self-adjoint on 
the dense domain $C_0^\infty(\R^2)  \subset L^2(\R^2)$ by \cite[Theorems 1 and 2]{LS}, and we still denote by $H$ its unique self-adjoint extension in 
$L^2(\R^2)$.

\subsection{Reduction to one-dimensional operators}

The Schr\"odinger operator defined in \eqref{a4} being invariant with respect to translations in the $y$-direction, it decomposes into 
a family of 
parameterized Hamiltonians on $L^2(\R)$. Let $\cF$ denote the partial Fourier transform with respect to $y$, i.e.
$$ (\cF u)(x,\xi) = \hat{u}(x,\xi) := (2 \pi)^{-1\slash 2} \lim_{N \to \infty} \int_{-N}^N  e^{-i\xi y} u(x,y) \dd y,\ u \in L^2(\R^2),\ (x,\xi) \in \R^2, $$
the limit being taken in the $L^2(\R)$-sense. We have $\hat{u} \in L^2(\R^2)$ and
\[
y \mapsto \hat{u}(\xi,y) \ind_{[-N,N]}(y) \in L^1(\R) \cap L^2(\R)\;  \mbox{for a.e.}\;  \xi \in \R, 
\]
where $\ind_{[-N,N]}$ stands for the characteristic function of $[-N,N]$, and we recall that the Fourier inversion formula reads 
$$ u(x,y) = (\cF^* \hat{u})(x,y) := (2 \pi)^{-1\slash 2} \lim_{N \to \infty} \int_{-N}^N e^{i\xi y} \hat{u}(x,\xi) \dd \xi,\  (x,y) \in \R^2. $$
The Hilbert space $L^2(\R^2)$ can be 
expressed as a constant fiber direct integral over $\R$ with fibers $L^2 (\R)$, i.e. $L^2(\R^2)=\int_\R^\oplus L^2(\R) \dd \xi$, and the operator $H$ admits a partial Fourier decomposition with respect to the $y$-variable, 
with
$$
\cF H \cF^* = \int_{\R}^\oplus h(\xi) \dd \xi,
$$
where each $h(\xi)$, $\xi \in \R$, is self-adjoint in $L^2(\R)$ with domain $D(h(\xi))$ which is independent of $\xi$, i.e. $D(h(\xi))=D(h(0))$, according to \cite[Lemma 2.3]{I}. Moreover, we have 
\bel{r1}
h(\xi)=-\frac{\dd^2}{\dd x^2}+q(x,\xi)\ \mbox{on}\ C_0^\infty (\R),
\ee
where $q(x,\xi):=v(x,\xi)^2 \slash 4$ and $v(x,\xi):=2(\xi-a(x))$ denotes the quantum velocity for the wave number $\xi$.  

In light of \eqref{a1}, the potential $q(\cdot, \xi)$, $\xi \in \R$, tends to infinity as $|x|$ goes to infinity, hence $h(\xi)$ has a compact resolvent. 
Let $\{ \lambda_j(\xi),\ j \in \N \}$, be the non decreasing sequence of the eigenvalues of the operator $h(\xi)$, $\xi \in \R$. 
Since all the eigenvalues $\lambda_j(\xi)$ are simple (see \cite[Proposition A2]{HS1} or \cite[Lemma 2.3]{I}), we have for all $j \ge 3$,
$$ \lambda_1(\xi) < \lambda_2(\xi) < \ldots < \lambda_j(\xi) < \lambda_{j+1}(\xi) < \ldots $$
and the functions $\xi \mapsto \lambda_j(\xi)$, $j \ge 1$, are real analytic by the Kato perturbation theory, see \cite[Chap. VII]{K}. Moreover we have
\bel{a5}
(2j-1) b_- \le \lambda_j(\xi) \le (2j-1) b_+,\ \xi \in \R,\ j \in \N,
\ee
and
\bel{a6} 
\lim_{\xi \to \pm \infty} \lambda_j(\xi) = (2j-1)b_\pm,\ j \in \N, 
\ee
from \cite[Proposition 3.1]{DGR}. As a consequence the spectrum of $H$ is purely absolutely continuous (see, e.g. \cite[Theorem XIII.86]{RS4}) and
$$
\sigma(H) = \bigcup_{j=1}^\infty [(2j-1)b_-,(2j-1)b_+].
$$
Therefore, $\sigma(H)$ has a band structure and it follows from \eqref{a2} that the first band $[b_-,b_+]$ does not overlap with the remaining part of the spectrum $\cup_{j=2}^\infty [(2j-1)b_-, (2j-1)b_+]$.

The main purpose of this article is to examine whether knowledge of the first band function $\lambda_1$ uniquely determines the magnetic field $b$ of the Iwatsuka Hamiltonian $H$.

\subsection{Quantum velocity}
In light of \cite[Lemma 2.3]{I}, there exists a $L^2(\R)$-orthonormal basis $\{ \varphi_j(\cdot,\xi),\ j \in \N \}$ of eigenfunctions of $h(\xi)$, $\xi \in \R$, such that
\bel{qv1} 
h(\xi) \phi_j(\cdot,\xi)=\lambda_j(\xi) \phi_j(\cdot,\xi),\ j \in \N.
\ee
Moreover, $\phi_j(\cdot,\xi) \in D(h(0)) = \{ u \in H^1(\R),\ -u'' + a^2 u \in L^2(\R) \}$, $j \in \N$, depends analytically on $\xi \in \R$ with respect to the graph norm of $h(0)$, defined by 
$$ \| u \|_{D(h(0))}:=\left( \| u \|_2^2 + \| h(0) u \|_2^2 \right)^{1 \slash 2},\ u \in D(h(0)), $$
where $\| \cdot \|_2$ denotes the usual norm in $L^2(\R)$. 
In what follows, all the eigenfunctions $\varphi_j(\cdot,\xi)$, $j \in \R$, are chosen to be real-valued, and since $\varphi_1(\cdot,\xi)$ {\color{red} has a constant sign}, we will always assume that
\bel{qv2} 
\varphi_1(x,\xi) >0,\ x \in \R.
\ee

This being said, we introduce the {\it current operator} $\vartheta$ as
\bel{c1} 
\vartheta(\chi) := \int_{\R} \chi(\xi)^2 \langle v(\cdot,\xi) \phi_1(\cdot,\xi), \phi_1(\cdot,\xi) \rangle_2 \dd \xi,\ \chi \in C_0^\infty(\R):=C_0^\infty(\R,\R),
\ee
where $\langle \cdot , \cdot \rangle_2$ is the usual scalar product in $L^2(\R)$.
Notice that $\vartheta$ is literally stated as a number here, which can be interpreted as a quadratic form in a second step and hence as an operator.
For this reason we keep calling $\vartheta$ an operator in the following although this is a little bit an abuse of langage. 
Further, we shall see in Section \ref{sec-tp}, that $\vartheta(\chi)$ is the expectation of the second component $2(i \pd_y+a)$ of the {\it velocity operator} $2(i \nabla+A)$ expressed in the quantum state $e^{-it H} u_{0,\chi}:=\cF^* \left( e^{-it \lambda_1} \psi_\chi \right)$, $t \in [0,+\infty)$, where
\bel{c1b} 
u_{0,\chi} := \cF^* \psi_\chi\ \mbox{and}\ \psi_\chi(x,\xi):= \chi(\xi) \phi_1(x,\xi),\ (x,\xi) \in \R^2.
\ee
{\color{red} This amounts to saying that} $\vartheta(\chi)$ is the quantum current carried by 
$e^{-it H} u_{0,\chi} $.
And since it is time-independent according to \eqref{c1}, we rather call it {\it quantum current carried by $u_{0,\chi}$} in the sequel.

In the present article we also investigate the inverse problem to know whether the Iwatsuka Hamiltonian $H$ can be retrieved from its transport properties, expressed through the current operator $\vartheta$.

\subsection{Main results and outline of the article}
\label{sec-mainres}

We denote by $\B$ the (convex) set of Iwatsuka magnetic fields, i.e. the set of non-decreasing functions $b \in L^\infty(\R)$ satisfying the condition \eqref{a1}-\eqref{a2}. Given $b$ and $\tilde{b}$ in $\B$, we introduce
$$ b_\ep:=(1-\ep) b+\ep \tilde{b},\ \ep \in (0,1]. $$
Evidently, $b_\ep \in \B$ for all $\ep \in (0,1]$ and we denote by $\{ \lambda_{j,\ep}(\xi),\ j \in \N \}$ the set of eigenvalues arranged in increasing order, of the operator
\bel{m0}
h_\ep(\xi):= -\frac{\dd^2}{\dd x^2} + q_\ep(\cdot,\xi),\ \xi \in \R,
\ee
where $q_\ep(\cdot,\xi):=v_\ep(\cdot,\xi)^2 \slash 4$, $v_\ep(x,\xi):=2 (\xi-a_\ep(x))$ and $a_\ep(x):=\int_0^x b_\ep(s) \dd s$. 

Our first identification result is as follows.

\begin{thm}
\label{thm1}
Put 
$$ r_0:= \frac{(3b_- - b_+)^{1 \slash 2}}{2b_+}. $$ 
Pick $b$ and $\tilde{b}$ in $\B$ such that
\bel{m1}
\supp (\tilde{a}-a) \subset [-r,r]
\ee
for some $r \in (0,r_0)$,
where $a$ is defined by \eqref{a3} and $\tilde{a}(x):=\int_0^x \tilde{b}(s) \dd s$ for all $x \in \R$.\\
Then, if $0$ is an accumulation point  of $\{\epsilon\in (0,1], \lambda_{1,\epsilon}(0)=\lambda_1(0)\}$, we have $\tilde{b}=b$.
\end{thm}

\begin{rmk}
As can be easily seen from the proof of this theorem, displayed in 
Section \ref{proof-thm1}, the result is actually slightly more general as the set $\{\epsilon\in (0,1], \lambda_{1,\epsilon}(0)=\lambda_1(0)\}$ may be replaced by $\{\epsilon \in (0,1], \lambda_{1,\epsilon}(\xi)=\lambda_1(\xi)\}$, where
$\xi \in \R$ is small enough.
\end{rmk}

For $\ep \in (0,1]$, we denote by $\vartheta_\ep$ the current operator associated with $b_\ep$. 
Otherwise stated, $\vartheta_\ep(\chi)$, for $\ep \in (0,1]$ and $\chi \in C_0^\infty(\R)$, is the quantum current carried by the 
state $u_{0,\ep,\chi}$, characterized by
$$\hat{u}_{0,\ep,\chi}(\cdot,\xi)=\chi(\xi) \phi_{1,\ep}(\cdot,\xi),\ \xi \in \R. $$ 
Here, $\{ \phi_{j,\ep}(\cdot,\xi),\ j \in \N \}$ is a
$L^2(\R)$-orthonormal basis of eigenfunctions of the operator $h_\ep(\xi)$, $\xi \in \R$, satisfying
$$ h_\ep(\xi) \varphi_{j,\ep}(\cdot,\xi)=\la_{j,\ep}(\xi) \varphi_{j,\ep}(\cdot,\xi),\ j \in \N. $$
As (it is fairly well known that) knowledge of the current operator $\vartheta$ yields knowledge of the first band function $\lambda_1$ (see Proposition \ref{pr1}), the following result is a byproduct of Theorem \ref{thm1}.

\begin{cor}
\label{cor1}
Under the conditions of Theorem \ref{thm1} we have $\tilde{b}=b$ whenever $0$ is an accumulation point of
$\{\epsilon \in (0,1], \vartheta_\ep = \vartheta \}$.
\end{cor}

In contrast to Theorem \ref{thm1} (resp., Corollary \ref{cor1}) where an infinite number of spectral 
data $\lambda_{1,\epsilon}(0)$ (resp., velocity data in the form of the current operators $\vartheta_\ep$), $\epsilon \in (0, 1]$, are supposed to be known, the coming result only assumes knowledge of the ground state for an arbitrary wave number. Namely, writing 
$(\tilde{\lambda}_j,\tilde{\varphi}_j)$, $j \in \N$, instead of $(\lambda_{j,1},\varphi_{j,1})$, that is to say that 
$\{ \tilde{\phi}_{j}(\cdot,\xi),\ j \in \N \}$ is a 
$L^2(\R)$-orthonormal basis of eigenfunctions of the operator 
$\tilde{h}(\xi):= h_1(\xi)$, such that
$$ \tilde{h}(\xi) \tilde{\varphi}_j(\cdot,\xi) = \tilde{\lambda}_j(\xi) \tilde{\varphi}_j(\cdot,\xi),\ j \in \N, $$
and that $\{ \tilde{\lambda}_j,\ j \in \N \}$ is the set of band functions associated with the operator 
$$\tilde{H}:=-\partial_x^2 + (-i \partial_y - \tilde{a}(x))^2, $$
we have the:
\begin{thm}
\label{thm2}
Let $b \in \B$ and $\tilde{b} \in \B$. If the condition
\bel{m6} 
\phi_1(x,\xi_0) = \tilde{\phi}_1(x, \xi_0),\ x \in \R,
\ee
holds for some $\xi_0 \in (0,\infty)$, then we have $\tilde{b}=b$.
\end{thm}

The remaining part of this article is structured as follows. In Section \ref{sec-data} we rigorously define the velocity data, i.e. the current operator $\vartheta$, used for solving the inverse problem under examination in Corollary \ref{cor1}, and we briefly comment on it. 
Finally, in Section \ref{sec-pr}, we give the proof of the main results stated in Theorems \ref{thm1} and \ref{thm2}, and in Corollary \ref{cor1}.

\section{Preliminaries: definition of the current operator $\vartheta$}
\label{sec-data}

In this section we study the transport properties of quantum devices described by the system
\bel{v1}
\left\{ \begin{array}{ll} (-i \partial_t + H) u(x,y,t)=0, & (x,y,t) \in \R^2 \times (0,\infty) \\ u(x,y,0)=u_0(x,y), & (x,y) \in \R^2, \end{array} \right.
\ee
where $H$ is the self-adjoint realization introduced in Section \ref{sec-settings} of the Iwatsuka Hamiltonian 
defined by \eqref{a4} on $C_0^\infty(\R^2)$, and $u_0$ is taken in $D(H)$, the domain of $H$. 
More precisely, we aim to relate the current operator $\vartheta$ defined in \eqref{c1b} to the second component of the quantum velocity operator 
associated with $H$.

\subsection{The forward problem: Energy concentration and fast decaying property}
\label{sec-fp}

For all  $u \in L^2(\R^2)$, we have 
$$ \hat{u}(x,\xi)=\sum_{j=1}^\infty u_j(\xi) \phi_j(x,\xi),\ (x,\xi) \in \R^2, $$
where $u_j(\xi):= \langle \hat{u}(\cdot,\xi), \phi_j(\cdot,\xi) \rangle_2$, since $\{ \phi_j(\cdot,\xi),\ j \in \N \}$, $\xi \in \R$, is an orthonormal basis of 
$L^2(\R)$. Thus, for all $\lambda \in (0,\infty)$ and all $u \in D(H)$, it holds true that
\begin{align*}
\| (\lambda + iH) u \|_{L^2(\R^2)}^2 & =  \sum_{j=1}^\infty \int_\R \abs{\lambda + i \lambda_j(\xi)}^2 \abs{u_j(\xi)}^2 \dd \xi \\
& =  \sum_{j=1}^\infty \int_\R (\lambda^2 + \lambda_j(\xi)^2) \abs{u_j(\xi)}^2 \dd \xi \\
& \ge  \lambda^2 \| u \|_{L^2(\R^2)}^2,
\end{align*}
and hence the operator $-iH$ is dissipative in $L^2(\R^2)$. Further, $\lambda+iH=i(H-i\lambda)$ being surjective since the spectrum of $H$ is
embedded in $[b_-,\infty)$, the operator $-iH$ is m-dissipative in 
$L^2(\R^2)$. Therefore, \eqref{v1}
admits a unique solution $u \in C^0([0,\infty), D(H)) \cap C^1([0,\infty),L^2(\R^2))$, which is expressed as
\bel{v1b}
u(x,y,t) = e^{-i t H} u_0(x,y),\ (x,y) \in \R^2,\ t \in [0,\infty), 
\ee
from \cite[Lemma 2.1]{CKS}, where
$$  e^{-i t H} u_0:= \cF^* \left( \sum_{j=1}^\infty e^{-i t \lambda_j} u_{0,j} \varphi_j \right) $$
and
$$ u_{0,j}(\xi):= \langle \hat{u}_0(\cdot,\xi), \phi_j(\cdot,\xi) \rangle_2,\ \xi \in \R. $$

For $\chi \in C_0^\infty(\R)$, let $u_{0,\chi}$ be the same as in \eqref{c1b}. Then we have $\hat{u}_{0,\chi} \in D(h(\xi))$ for all $\xi \in \R$, and
$$ \int_{\R} \left( \| \hat{u}_{0,\chi}(\cdot,\xi) \|_2^2 + \| h(\xi) \hat{u}_{0,\chi}(\cdot,\xi) \|_2^2 \right) \dd \xi 
= \int_{\R} \left( 1 + \lambda_1(\xi)^2 \right) \chi(\xi)^2 \dd \xi < \infty,
$$
whence $u_{0,\chi} \in D(H)$ by \cite[Section XIII.16]{RS4}. Therefore, it follows from \eqref{v1b} that
\bel{v1c} 
u_\chi(x,y,t):=e^{-it H} u_{0,\chi}(x,y),\ (x,y) \in \R^2,\ t \in [0,\infty),
\ee
is well-defined.

For further reference we shall establish that, 1) the quantum state $u_\chi(\cdot,\cdot,t)$, $t \in [0,\infty)$, has energy concentration in the first spectral band $(b_-,b_+)$, of $H$, and 2) $\partial_t^k u(\cdot,\cdot,t)$, $k=0,1$, together with its partial derivatives with respect to $y$, decay faster than any polynomials in the 
$y$-direction. For this purpose we introduce the Schwartz space $\cS_y(\R,L_x^2(\R))$ of smooth functions 
$y \mapsto f(\cdot,y)$ from 
$\R$ into $L^2(\R)$, whose derivatives are rapidly decreasing, as:
$$ \cS_y(\R,L_x^2(\R)) := 
\{ f \in C_y^\infty(\R,L_x^2(\R)),\ \forall (m,n) \in \N_0^2,\ \sup_{y \in \R} \abs{y}^m \| \partial_y^n f (\cdot,y)\|_2<\infty \}, $$
where $\N_0:=\{ 0 \} \cup \N$. Next, we recall that 
the spectral projection of $H$ associated with a Borel set $I \subset \R$, reads
\bel{v3b} 
\proj_I w(x,y)= \cF^* \left( \sum_{j=1}^\infty \ind_{\lambda_j^{-1}(I)} w_j \phi_j(x,\cdot) \right)(y),\ w \in L^2(\R^2),
\ee
where $\ind_I$ is the characteristic function of $I$ and $w_j(\xi):=\langle \hat{w}(\cdot,\xi), \phi_j(\cdot,\xi) \rangle_2$.
Then, the expected result is as follows.

\begin{lem}
\label{lm2}
Let $u_\chi$, $\chi \in C_0^\infty(\R)$, be defined by \eqref{v1c}, where $u_{0,\chi}$ is as in \eqref{c1b}. Then, we have
\bel{v3c}
u_\chi \in C^1([0,\infty),\cS_y(\R,L_x^2(\R)))
\ee
and
\bel{v3d}
\proj_{(b_-,b_+)} u_\chi(\cdot,\cdot,t) = u_\chi(\cdot,\cdot,t),\ t \in [0,\infty).
\ee
\end{lem}
\begin{proof}
We start by proving \eqref{v3c}. To this end, we infer from \eqref{c1b} and \eqref{v1c} that 
$$\hat{u}_\chi(x,\xi,t)=e^{-i t \lambda_1(\xi)} \chi(\xi) \phi_1(x,\xi),\ (x,\xi) \in \R^2,\ t \in [0,\infty). $$  
Therefore, for a.e. $x \in \R$ and all $t \in [0,\infty)$, it holds true that $\xi \mapsto \hat{u}_\chi(x,\xi,t) \in C^\infty(\R)$.
Further, using that $\| \varphi_1(\cdot,\xi) \|_2=1$ for all $\xi \in \R$, we get for all $t \in [0,\infty)$ that
$\| \hat{u}_\chi(\cdot,\xi,t) \|_2 = \abs{\chi(\xi)}$ and that
$\| \partial_t \hat{u}_\chi(\cdot,\xi,t) \|_2 =\lambda_1(\xi)\abs{\chi(\xi)}$. As a consequence we have
$$\hat{u}_\chi \in C^1([0,\infty),\cS_\xi(\R,L_x^2(\R))),$$ 
and \eqref{v3c} follows from this since the partial 
Fourier transform $\cF$ is an automorphism of the Schwartz space $\cS(\R)$. As for \eqref{v3d}, this a straightforward consequence of 
\eqref{c1b} and \eqref{v1c}- \eqref{v3b}, because we have $\lambda_1^{-1}(b_-,b_+)=\R$ by virtue of \eqref{a5}-\eqref{a6}.
\end{proof}

\subsection{Quantum velocity}
\label{sec-tp}
Let $u$ be given by \eqref{v1b}. Assume moreover that 
$$u \in C^1([0,\infty),\cS_y(\R,L_x^2(\R))). $$
Then, the expectation of the $y$-component of the velocity operator of the system in the quantum state $u$ is (well-) defined by
$$ \upsilon(u_0,t) := \frac{\dd}{\dd t} \langle y u(\cdot,t), u(\cdot,t) \rangle_{L^2(\R^2)},\ t \in [0,\infty), $$ 
where $\langle \cdot , \cdot \rangle_{L^2(\R^2)}$ is the usual scalar product in $L^2(\R^2)$ and the notation $y$ stands for the multiplication operator by $y$. Otherwise stated, $\upsilon(u_0,t)$ is the velocity, i.e. the first 
time derivative, of the quantum realization in the state $e^{-i t H} u_0$, of the second component $y$ of the position observable. Hence 
$\upsilon(u_0,t)$ can be interpreted as the quantum current flowing in the $y$-direction, that is carried by the state $e^{-i t H} u_0$. 
We refer the reader to \cite{DGR, DHS, HS3} and the references therein, for an extensive mathematical study of the transport properties
of Iwatsuka Hamiltonians.

Further, since $u=e^{-it H} u_0  \in C^1([0,\infty),\cS_y(\R,L_x^2(\R)))$ yields that
\bel{v1e}  
H u \in C^0([0,\infty),\cS_y(\R,L_x^2(\R))),
\ee
we see that 
\begin{align*} 
\upsilon(u_0,t) & =  \frac{\dd }{\dd t} \langle y e^{-it H} u_0, e^{-it H} u_0 \rangle_{L^2(\R^2)}\\
& =  -i \left( \langle y H e^{-it H} u_0 , e^{-it H} u_0 \rangle_{L^2(\R^2)} - \langle y e^{-it H} u_0 , H e^{-it H} u_0 \rangle_{L^2(\R^2)} \right)\\
& =  -i \langle [y, H] e^{-it H} u_0 , e^{-it H} u_0 \rangle_{L^2(\R^2)},
\end{align*}
where $[y,H]$ denotes the commutator of $y$ with $H$. Next, using that $[y,H]=[y,(-i\partial_y-a)^2]$ and that $[y, (-i\partial_y-a)]=i$, we find that
$[y,H]=2i(-i \partial_y -a)$, and hence that
\bel{v2} 
\upsilon(u_0,t) = 2 \langle (-i \partial_y  - a) e^{-it H} u_0 , e^{-it H} u_0 \rangle_{L^2(\R^2)},\ t \in [0,\infty).
\ee
Notice from \eqref{v1e} that $(-i \partial_y  - a) e^{-it H} u_0 \in L^2(\R^2)$ and hence that the right-hand side of \eqref{v2} is well-defined, as we have
$$ \| \partial_x e^{-it H} u_0 \|_{L^2(\R^2)}^2 + \| (-i \partial_y  - a) e^{-it H} u_0 \|_{L^2(\R^2)}^2 = 
\langle H e^{-it H} u_0 , e^{-it H} u_0 \rangle_{L^2(\R^2)}. $$
Now, the transform $\cF$ being unitary in $L^2(\R^2)$, we deduce from the identity 
$$2\cF (-i \partial_y -a(x)) \cF^* = 2(\xi-a(x)) = v(x,\xi),\ (x,\xi) \in \R^2, $$
and from \eqref{v2} that
\bel{v3} 
\upsilon(u_0,t) = \langle v \cF (e^{-it H} u_0) , \cF (e^{-it H} u_0)\rangle_{L^2(\R^2)},\ t \in [0,\infty).
\ee

Let us now express $\upsilon(u_0,t)$ when $u_0 \in \proj_I (L^2(\R^2))$ for some $I \subset \R$, that is to say when $u_0=\proj_I u_0$. 
In this case, we have   
$$ \cF (e^{-it H} u_0)(x,\xi) = \sum_{j=1}^\infty \ind_{\lambda_j^{-1}(I)}(\xi) e^{-it \lambda_j(\xi)} u_{0,j}(\xi) \phi_j(x,\xi),\ (x,\xi) \in \R^2, $$
from \eqref{v3b}, where $u_{0,j}(\xi):=\langle \hat{u}_0(\cdot,\xi), \phi_j(\cdot,\xi) \rangle_2$. Putting this into \eqref{v3}, we obtain that
\begin{align}
&\upsilon(u_0,t)\label{v4}
\\
&\qquad = \sum_{j,k=1}^\infty \int_{\lambda_j^{-1}(I) \cap \lambda_k^{-1}(I)} e^{-it(\lambda_j(\xi)-\lambda_k(\xi))} u_{0,j}(\xi) \overline{u_{0,k}(\xi)}\langle v(\cdot,\xi) \phi_j(\cdot,\xi) , \phi_k(\cdot,\xi) \rangle_2 \dd \xi.\nonumber
\end{align}
Now, suppose that $I \subset (b_-,b_+)$, in such a way that we have
$$\lambda_j^{-1}(I)=\emptyset,\ j \geq 2, $$
by virtue of \eqref{a2} and \eqref{a5}. Then, it follows from \eqref{v4} that
\bel{v5}
\upsilon(u_0,t) = \int_{\lambda_1^{-1}(I)} \abs{u_{0,1}(\xi)}^2 \langle v(\cdot,\xi) \phi_1(\cdot,\xi) , \phi_1(\cdot,\xi) \rangle_2 \dd \xi .
\ee
Therefore, the quantum current $\upsilon(u_0,t)$ carried by a state $u_0$ with energy concentration in $I \subset (b_-,b_+)$, is independent of $t$. For the sake of notational simplicity, we write $\upsilon(u_0)$ instead of $\upsilon(u_0,t)$ in the following. 

Now, with reference to Lemma \ref{lm2}, we deduce from \eqref{c1b} and \eqref{v5} that $\vartheta(\chi)$ is the quantum current carried by $e^{-it H} u_{0,\chi}$, $t \in [0,\infty)$, or equivalently by $u_{0,\chi}$:
$$ \vartheta(\chi) = \upsilon(u_{0,\chi}),\ \chi \in C_0^\infty(\R). $$

\subsection{Comments on the data $\vartheta$ and the formulation of the inverse problem}
\label{sec-comment}
It might seem surprising at first sight that we probe the system \eqref{v1} with the wave number profile $\chi$ of the initial state $u_{0,\chi}$, rather than with the initial state itself. But there are at least two reasons why the current operator $\vartheta$ should not be defined as a function of the initial states $u_0$, the first one being that this would be physically irrelevant. Indeed, since any initial state of the quantum system \eqref{v1} governed by the Iwatsuka Hamiltonian $H$, 
with energy concentration between $b_-$ and $b_+$, is expressed as $\cF^* (\chi \phi_1)$ for some suitable $L^2(\R)$-function $\chi$ of the variable $\xi \in \R$, and since all the ground states
$\phi_1(\cdot,\xi)$, $\xi \in \R$, are determined by $H$, it is clear that only $\chi$ can be prescribed. 

Secondly, it turns out that from a mathematical viewpoint, the inverse problem of recovering the magnetic potential $a$ by triggering the system \eqref{v1} with a suitable set of initial states $u_0$, is pointless. This can be understood (upon using Proposition \ref{pr1} below) from the following lines. 

Given two magnetic fields 
$b$ and $\tilde{b} $ in $\mathbb B$, we aim to compare the quantum currents $\upsilon(u_0)$ and $\tilde{\upsilon}(u_0)$ induced by  \eqref{v1} associated with, respectively, $b$ and $\tilde{b}$, and endowed with a non-zero initial state $u_0 \in U$, where $U:=\proj_{(b_-,b_+)}(L^2(\R^2)) \cap \tilde{\proj}_{(b_-,b_+)}(L^2(\R^2))$. Here, $\tilde{\proj}_{(b_-,b_+)}$ denotes the spectral projection on $(b_-,b_+)$ of the Iwatsuka Hamiltonian $\tilde{H}$ obtained upon substituting $\tilde{b}$ for $b$ in \eqref{a3}-\eqref{a4}.
If such a state exists, that is to say if there exists $u_0 \in L^2(\R^2) \setminus \{ 0 \}$ such that
$$u_0=\proj_{(b_-,b_+)} u_0=\tilde{\proj}_{(b_-,b_+)} u_0, $$ 
then we have
\bel{c1c} 
\hat{u}_0(x,\xi)=u_{0,1}(\xi) \phi_1(x,\xi) =\tilde{u}_{0,1}(\xi) \tilde{\phi}_1(x,\xi),\ (x,\xi) \in \R^2, 
\ee
where $u_{0,1}(\xi):=\langle \hat{u}_0(\cdot,\xi) , \phi_1(\cdot,\xi) \rangle_2$, $\tilde{u}_{0,1}(\xi):=\langle \hat{u}_0(\cdot,\xi) , \tilde{\phi}_1(\cdot,\xi) \rangle_2$, and $\tilde{\varphi}_1(\cdot,\xi)$ is defined as in Theorem \ref{thm2}. Here, we recall that the notation $\langle \cdot, \cdot \rangle_2$ stands for the usual scalar product in $L^2(/R)$. Thus, upon squaring both sides of \eqref{c1c} and then integrating with respect to $x$ over $\R$, we get that $u_{0,1}(\xi)^2=\tilde{u}_{0,1}(\xi)^2$ for a.e. $\xi \in \R$, and hence that $\abs{u_{0,1}(\xi)}=\abs{\tilde{u}_{0,1}(\xi)}$. Therefore, the set 
$$\cN(u_0):=u_{0,1}^{-1}(\{0\})=\{\xi \in \R,\ u_{0,1}(\xi)=0 \}$$ 
can be equivalently defined as
$\cN(u_0):=\tilde{u}_{0,1}^{-1}(\{0\})=\{\xi \in \R,\ \tilde{u}_{0,1}(\xi)=0 \}$, and it holds true that
$\abs{\phi_1(\cdot,\xi)}=\abs{\tilde{\phi}_1(\cdot,\xi)}$ a.e. in $\R$, whenever $\xi \in \R \setminus \cN(u_0)$. As a consequence we have
\bel{c2}
\phi_1(x,\xi)=\tilde{\phi}_1(x,\xi),\ (x,\xi) \in \R \times (\R \setminus \cN(u_0)),
\ee
since $\phi_1(\cdot,\xi)$ and $\tilde{\phi}_1(\cdot,\xi)$ are positive functions for all $\xi \in \R$. 

Let us now assume for a while that the function
$$F(\xi) := \norm{\phi_1(\cdot,\xi)-\tilde{\phi}_1(\cdot,\xi)}_{L^2(\R)}^2,\ \xi \in \R,$$ 
is not identically zero. Since $F$ is real analytic and the Lebesgue measure of the zero set of a non-trivial real analytic function is zero, the Lebesgue measure of $\R \setminus \cN(u_0)$ should be zero, according to \eqref{c2}. This would mean that $u_{0,1}(\xi)=0$ for a.e. $\xi \in \R$, and hence that $\hat{u}_0=0$ in $L^2(\R^2)$, according to \eqref{c1c}, which is contradiction the fact that $u_0$ is non-zero.
Therefore, we have $F(\xi)=0$ for all $\xi \in \R$, and consequently $\phi_1(\cdot,\xi)=\tilde{\phi}_1(\cdot,\xi)$ in $L^2(\R)$ for all $\xi \in \R$.

Summing up, we have proved that the following equivalence holds:
\bel{c3} 
U \neq \{ 0 \} \Longleftrightarrow \forall \xi \in \R,\ \phi_1(\cdot,\xi)=\tilde{\phi}_1(\cdot,\xi)\ \mbox{in}\ L^2(\R).
\ee
Having seen this, let us suppose that $U \neq \{ 0 \}$, and assume in addition that
$$ \upsilon(u_0)=\tilde{\upsilon}(u_0),\ u_0 \in U. $$
Then, we have $\phi_1(\cdot,\xi)=\tilde{\phi}_1(\cdot,\xi)$ for all $\xi \in \R$, from \eqref{c3}, and it is clear for all $\chi \in C_0^\infty(\R)$ that
$u_0=\cF^*(\chi \phi_1) = \cF^*(\chi \tilde{\phi}_1) \in U$.  Moreover, since $\upsilon(u_0)=\vartheta(\chi)$ and $\tilde{\upsilon}(u_0)=\tilde{\vartheta}(\chi)$, where $\tilde{\vartheta}$ is the same as in Theorem \ref{thm2}, we get that
$$ \vartheta(\chi)=\tilde{\vartheta}(\chi),\ \chi \in C_0^\infty(\R). $$
Therefore, we have $\lambda_1=\tilde{\lambda}_1$ by Proposition \ref{pr1}, whence
$$ \left( (\xi-\tilde{a}(x))^2 - (\xi-a(x))^2 \right) \phi_1(x,\xi)=0,\ (x,\xi) \in \R^2. $$
Since $\phi_1(\cdot,\xi)$ is positive for all $\xi \in \R$, this entails that 
$$ (\xi-\tilde{a}(x))^2 = (\xi-a(x))^2,\ (x,\xi) \in \R^2. $$
Upon differentiating the above identity with respect to $\xi$, we get that $\xi-\tilde{a}(x) = \xi-a(x)$ for all $(x,\xi) \in \R^2$, and hence that $a=\tilde{a}$ in $\R$.

\section{Analysis of the inverse problem}
\label{sec-pr}
We start with a technical result needed by the proof of Corollary \ref{cor1}.

\subsection{Proof of Corollary \ref{cor1}}
We aim to establish that knowledge of the current operator uniquely determines the first band function. With reference to the notations of Section \ref{sec-mainres}, the corresponding result can be stated as follows.

\begin{prop}
\label{pr1}
Let $b$ and $\tilde{b}$ be in $\mathbb{B}$. Assume that $\vartheta(\chi)=\tilde{\vartheta}(\chi)$ for all $\chi \in C_0^\infty(\R)$.
Then it holds true that 
$$ \lambda_1(\xi)=\tilde{\lambda}_1(\xi),\ \xi \in \R. $$
\end{prop}
\begin{proof}
Applying \eqref{v5} with $I=(b_-,b_+)$ and $u_{0,1}=\chi$, we obtain that
$$ \vartheta(\chi) = \int_{\R} \chi(\xi)^2 \langle v(\cdot,\xi) \phi_1(\cdot,\xi) , \phi_1(\cdot,\xi) \rangle_2 \dd \xi, $$
from \eqref{a5}-\eqref{a6}. Further, since $v(\cdot,\xi)=2(\xi-a)$ is the (formal) derivative of $h(\xi)$ with respect to $\xi$, we have
$$ \langle v(\cdot,\xi) \phi_1(\cdot,\xi) , \phi_1(\cdot,\xi) \rangle_2= \lambda_1'(\xi),\ \xi \in \R, $$
by the Feynman-Hellmann theorem 
(see, e.g. \cite[Chapter VII, Problem 4.19]{K}), and hence
$$ \vartheta(\chi)  = \int_{\R} \chi(\xi)^2 \lambda_1'(\xi) \dd \xi.$$
Similarly, we have $\tilde{\vartheta}(\chi)  = \int_{\R} \chi(\xi)^2 \tilde{\lambda}_1'(\xi) \dd \xi$, and consequently
\bel{p1}
\int_{\R} (\lambda_1-\tilde{\lambda}_1)'(\xi) \chi(\xi)^2 \dd \xi =0,\ \chi \in C_0^\infty(\R), 
\ee
directly from the assumption. 

Having seen this, we will prove by contradiction that $(\lambda_1-\tilde{\lambda}_1)'$ is identically zero in $\R$. For this purpose we
assume existence of $\xi_0 \in \R$ such that $(\lambda_1-\tilde{\lambda}_1)'(\xi_0) \neq 0$. Since $\lambda_1$ and $\tilde{\lambda}_1$ play symmetric roles here, we may assume without loss of generality that $\delta:=(\lambda_1-\tilde{\lambda}_1)'(\xi_0) > 0$. Thus, by continuity of $\xi \mapsto (\lambda_1-\tilde{\lambda}_1)'(\xi)$ at $\xi_0$, there exists $\epsilon>0$ such that $(\lambda_1-\tilde{\lambda}_1)'(\xi)\ge \delta \slash 2$ whenever 
$\abs{\xi-\xi_0} \le \varepsilon$. As a consequence we have
$$ \int_{\R} (\lambda_1-\tilde{\lambda}_1)'(\xi) \chi^2(\xi) \dd \xi \ge \frac{\delta}{2} \int_\R \chi(\xi)^2 \dd \xi >0, $$
for all $\chi \in C_0^\infty(\R) \setminus \{ 0 \}$ that is supported in $(\xi_0-\varepsilon,\xi_0+\varepsilon)$. This contradicts \eqref{p1} and shows that 
$(\lambda_1-\tilde{\lambda}_1)'(\xi)=0$ for all $\xi \in \R$. Therefore, there exists $C \in \R$ such that
$$ \lambda_1(\xi)=\tilde{\lambda}_1(\xi)+C,\ \xi \in \R. $$
Now, since $\lambda_1$ and $\tilde{\lambda}_1$ fulfill \eqref{a6}, we get that $C=0$ upon 
sending $\xi$ to infinity in the above identity, and the result follows.
\end{proof}
In view of Proposition \ref{pr1}, the statement of Corollary \ref{cor1} follows directly from Theorem \ref{thm1}.

\subsection{Proof of Theorem \ref{thm1}}
\label{proof-thm1}
The proof being quite lengthy, we split it into six steps.

\noindent{\it Step 1: Spectral projections.}
Let us denote by $r_\xi$, $\xi \in \R$, the resolvent operator of $h(\xi)$, i.e., 
$$r_\xi(z):=(h(\xi)-z)^{-1},\ z \in \C \setminus \{ \lambda_j(\xi),\ j \in \N \}. $$ 
Then, the spectral projection of $h(\xi)$ associated with $\lambda_1(\xi)$ can be expressed as
\bel{ap3}
p_1(\xi)=-\frac{1}{2 i \pi} \int_{C(\lambda_1(\xi),\rho)} r_\xi(z) \dd z, 
\ee
where $\rho$ is arbitrarily fixed in $(0, 3b_- -b_+)$ and $C(\lambda_1(\xi),\rho):=\{ \lambda_1(\xi)+ \rho e^{i \theta},\ \theta \in [0, 2 \pi) \}$ is the circle centered at $\lambda_1(\xi)$ with radius $\rho$, oriented counterclockwise, see e.g., \cite[Section VII.3, Eq. 1.3]{K}. With reference to \eqref{m1} and the notations introduced in the following line, we have $q_\ep(x,\xi)=q(x,\xi)+ \ell_\ep(x,\xi)$ for all $(x,\xi) \in \R^2$, where
\bel{ap3b} 
\ell_\ep(x,\xi):= \ep \omega(x,\xi) + \ep^2 w(x)^2,\ \omega(x,\xi):= -v(x,\xi) w(x)\ \mbox{and}\ w(x):=\tilde{a}(x)-a(x).
\ee
Thus, putting $\delta:= \norm{\tilde{a}-a}_{L^\infty(\R)}$ and $M(\xi):=\norm{v(\cdot,\xi)}_{L^\infty(K)} = 2 \norm{\xi-a}_{L^\infty(K)}<\infty$, where $K:=[-r,r]$, we infer from \eqref{ap3b} that 
\bel{ap3c}
\norm{\ell_\ep(\cdot,\xi)}_{L^\infty(\R)}\le \ep C(\xi),\ C(\xi):=\delta (M(\xi)+\delta),\ \ep \in (0,1),\ \xi \in \R.
\ee
From this and the MinMax principle, it then follows that
\bel{ap3d}
\abs{\lambda_1(\xi)-\lambda_{1,\ep}(\xi)} \leq \ep C(\xi),\ \ep \in (0,1),\ \xi \in \R,
\ee
where $\lambda_{1,\ep}(\xi)$ denotes the first eigenvalue of the operator $h_\ep(\xi)$. 
Therefore, writing $\lambda_{j,\ep}(\xi)$ for the $j$-th eigenvalue, $j \in \N$, of $h_\ep(\xi)$ and taking $\ep \in (0,1)$ so small that $\ep C(\xi) < 3b_- - b_+ - \rho$, we deduce from \eqref{ap3d} that
\bel{ap3e}
\overline{D}(\lambda_1(\xi),\rho) \cap \{ \lambda_{j,\ep}(\xi), j \in \N \} = \{ \lambda_{1,\ep}(\xi) \},
\ee
where $\overline{D}(\lambda_1(\xi),\rho):=\{ z \in \C, \abs{z -\lambda_1(\xi)} \leq \rho \}$.

Set $r_{\ep,\xi}(z):=(h_\ep(\xi)-z)^{-1}$ for all $z \in \C \setminus \{ \lambda_{j,\ep}(\xi), j \in \N \}$, and put
$\ep_*:=\ep_*(\xi,\rho) = \min(1, C(\xi)^{-1} (3b_- - b_+ - \rho))$. Then, with reference to \eqref{ap3e} and 
the path independence of contour integration of the meromorphic function $z \mapsto r_{\ep,\xi}(z)$ around $\lambda_{1,\ep}(\xi)$, the spectral projection of $h_\ep(\xi)$ associated with
$\lambda_{1,\ep}(\xi)$, 
$$p_{1,\ep}(\xi)=-\frac{1}{2 i \pi} \int_{C(\lambda_{1,\ep}(\xi),\rho)} r_{\ep,\xi}(z) \dd z,\ 
$$
can be equivalently rewritten as
\bel{ap5}
p_{1,\ep}(\xi)=-\frac{1}{2 i \pi} \int_{C(\lambda_1(\xi),\rho)} r_{\ep,\xi}(z) \dd z,\ \xi \in \R,\ \ep \in (0,\ep_*).
\ee

\noindent {\it Step 2: Relating $\lambda_{1,\ep}$ to $\lambda_1$.}
Having established \eqref{ap5}, we turn now to relating $\lambda_{1,\ep}(\xi)$ to $\lambda_1(\xi)$ with the aid the identity $h_{\ep}(\xi)=h(\xi)+\ell_\ep(\cdot,\xi)$.
To do that, we start from the eigenvalue equality $h_{\ep}(\xi) p_{1,\ep}(\xi) \phi_1(\cdot,\xi)=\lambda_{1,\ep}(\xi) p_{1,\ep}(\xi) \phi_1(\cdot,\xi)$, recall  
that the operators $h(\xi)$ and $\ell_\ep(\cdot,\xi)$ are self-adjoint in $L^2(\R)$, and obtain for all $\xi \in \R$ and all $\ep \in (0,1)$, that 
\begin{align*}
& \lambda_{1,\ep}(\xi)  \langle p_{1,\ep}(\xi) \phi_1(\cdot,\xi), \phi_1(\cdot,\xi) \rangle_2 \\
&\hskip 2cm= \langle h_{\ep}(\xi) p_{1,\ep}(\xi) \phi_1(\cdot,\xi) ,\phi_1(\cdot,\xi) \rangle_2\\
&\hskip 2cm =  \langle h(\xi) p_{1,\ep}(\xi) \phi_1(\cdot,\xi) ,\phi_1(\cdot,\xi) \rangle_2+\langle \ell_\ep(\xi) p_{1,\ep}(\xi) \phi_1(\cdot,\xi) ,\phi_1(\cdot,\xi) \rangle_2 \\
&\hskip 2cm =  \langle p_{1,\ep}(\xi) \phi_1(\cdot,\xi) , h(\xi) \phi_1(\cdot,\xi) \rangle_2+ \langle  p_{1,\ep}(\xi) \phi_1(\cdot,\xi) , \ell_\ep(\xi)\phi_1(\cdot,\xi) \rangle_2 \\
&\hskip 2cm =  \lambda_1(\xi) \langle p_{1,\ep}(\xi) \phi_1(\cdot,\xi) ,  \phi_1(\cdot,\xi) \rangle_2+ \langle  p_{1,\ep}(\xi) \phi_1(\cdot,\xi) , \ell_\ep(\xi)\phi_1(\cdot,\xi) \rangle_2.
\end{align*}
Therefore, for all $\xi \in \R$ and all $\ep \in (0,1)$, we have 
\begin{align}
F_\xi(\ep) :&=  \langle  p_{1,\ep}(\xi) \phi_1(\cdot,\xi) , \ell_\ep(\xi)\phi_1(\cdot,\xi) \rangle_2 \label{ap5b} \\
& =  (\lambda_{1,\ep}(\xi) - \lambda_1(\xi)) \langle p_{1,\ep}(\xi) \phi_1(\cdot,\xi), \phi_1(\cdot,\xi) \rangle_2.\nonumber
\end{align}

\noindent{\it Step 3: Resolvent formula.} 
With reference to the second resolvent formula, we have
\bel{ap7} 
r_{\ep,\xi}(z)=r_\xi(z)-r_\xi(z) \ell_\ep(\cdot,\xi) r_{\ep,\xi}(z),\ \ep \in (0,\ep_*),\ \xi \in \R,\ z \in C(\lambda_1(\xi),\rho).
\ee
Notice that for all $\xi \in \R$ and all $z \in C(\lambda_1(\xi),\rho)$, we have
$$
\mathrm{dist} \left( z,\{\lambda_j(\xi),\ j \in \N \} \right)= \min (\abs{\lambda_1(\xi)-z}, \abs{\lambda_2(\xi)-z}) \\
\ge \min(\rho,3b_- - b_+-\rho),$$
whence
\begin{align*}
\norm{\ell_\ep(\cdot,\xi) r_\xi(z)}_{\cB(L^2(\R))}&  \le  \norm{\ell_\ep(\cdot,\xi)}_{L^\infty(\R)} \norm{r_\xi(z)}_{\cB(L^2(\R))} \\
& \le   \frac{\ep C(\xi)}{\min(\rho,3b_- - b_+-\rho)} ,
\end{align*}
from \eqref{ap3c}, where $\cB(L^2(\R)$ denotes the space of linear bounded operators in $L^2(\R)$. Therefore, we get that $\norm{\ell_\ep(\cdot,\xi) r_\xi(z)}_{\cB(L^2(\R))}<1$ for all $\xi \in \R$ and all $z \in C(\lambda_1(\xi),\rho)$, provided that 
$$ \ep \in (0,\ep_\star),\ \ep_\star=\ep_\star(\xi,\rho) := \min(\ep_*, C(\xi)^{-1}\rho). $$
Thus, by iterating \eqref{ap7} we get for all $\xi \in \R$ and all $z \in C(\lambda_1(\xi),\rho)$, that
$$ r_{\ep,\xi}(z)= \sum_{n=0}^\infty (-1)^n r_\xi(z) \left( \ell_\ep(\cdot,\xi) r_\xi(z) \right)^n,\ \ep \in (0,\ep_\star), $$
where the series converges in $\cB(L^2(\R))$. In view of \eqref{ap3b}, this leads to
\bel{ap8}
r_{\ep,\xi}(z) 
=\sum_{n=0}^\infty (-1)^n \theta_{n,\xi}(z) \ep^n ,\ \xi \in \R,\ \ep \in (0,\ep_\star),
\ee
the series being convergent in $\cB(L^2(\R))$, uniformly in $z \in C(\lambda_1(\xi),\rho)$. 
Here, each $\theta_{n,\xi}(z) \in \cB(L^2(\R))$, $n \in \N \cup \{0 \}$, can be expressed in terms of $r_\xi(z)$, $\omega(\cdot,\xi)$ and $w$ only. For instance, we get through elementary computations that 
\begin{align}
\theta_{0,\xi}(z)  & =   r_\xi(z), \label{ap4-a} \\
\theta_{1,\xi}(z)  & =   r_\xi(z) \omega(\cdot,\xi) r_\xi(z), \label{ap4-b} \\
\theta_{2,\xi}(z) & =   r_\xi(z) ( w^2 r_\xi(z)- (\omega(\cdot,\xi) r_\xi(z))^2).
\label{ap4-c}
\end{align}
\noindent{\it Step 4: Analytic expansion of $F_\xi$.}
By inserting \eqref{ap8} into \eqref{ap5}, we obtain with the aid of \eqref{ap3} that
\begin{align*}
p_{1,\ep}(\xi) & =  \sum_{n=0}^\infty \frac{(-1)^{n+1}}{2i \pi} \left( \int_{C(\lambda_1(\xi),\rho)} \theta_{n,\xi}(z) \dd z \right) \ep^n \\
& =  p_1(\xi) +  \sum_{n=1}^\infty \frac{(-1)^{n+1}}{2i \pi} \left( \int_{C(\lambda_1(\xi),\rho)} \theta_{n,\xi}(z) \dd z \right) \ep^n,\ \xi \in\R,\ \ep \in (0,\ep_\star).
\end{align*}
We recall from \eqref{ap5b} that $\ep \mapsto F_\xi(\ep)$ is defined as the $L^2(\R)$-scalar product of $p_{1,\ep}(\xi) \varphi_1(\cdot,\xi)$ with $\ell_\ep(\cdot,\xi) \varphi_1(\cdot,\xi)$. Moreover, since $\ell_\ep(\cdot,\xi)$ is a polynomial function in $\ep$ according to \eqref{ap3b}, it follows readily from the above equality that $F_\xi(\ep)$ can be brought into the form
\bel{ap8b} 
F_\xi(\ep)=\sum_{n=1}^\infty A_n(\xi) \ep^n,\ \xi \in \R,\ \ep \in (0,\ep_\star), 
\ee
where the series is convergent in $(0,\ep_\star)$, uniformly in $\xi \in \R$, and each $A_n : \R \to \C$, $n \in \N$, is independent of $\ep$. As a matter of fact, in the special case when $n=2$, we find with the aid of \eqref{ap4-a} --\eqref{ap4-c} by elementary calculation, that for all $\xi \in \R$,
\begin{align}
A_2(\xi) & =  \langle p_1(\xi) \phi_1(\cdot,\xi) , w^2 \phi_1(\cdot,\xi) \rangle_2 \label{ap8c}
\\
& \qquad + \frac{1}{2i \pi} \int_{C(\lambda_1(\xi),\rho)} \langle r_\xi(z) \omega(\cdot,\xi) r_\xi(z) \phi_1(\cdot,\xi)  , \omega(\cdot,\xi) \phi_1(\cdot,\xi) \rangle_2 \dd z \nonumber \\
& =  \| w \phi_1(\cdot,\xi) \|_2^2 + \frac{1}{2i \pi} \int_{C(\lambda_1(\xi),\rho)}  \langle ( \omega(\cdot,\xi) r_\xi(z) )^2 \phi_1(\cdot,\xi)  , \phi_1(\cdot,\xi) \rangle_2 \dd z. \nonumber
\end{align}
Evidently, for all $\xi \in \R$ fixed, the sum on the right-hand side of \eqref{ap8b} extends to $\ep \in (-\ep_\star,\ep_\star)$ since the convergence radius of the series is at least $\ep_\star$. Hence, each $F_\xi$ can be extended to a real-analytic function at 
$\ep=0$. Moreover, $\ep=0$ is, by assumption, an accumulation point of the zeros of the function
$\ep \mapsto \la_{1,\ep}(\xi)-\la(\xi)$, hence the same is true for $\ep \mapsto F_\xi(\ep)$ from \eqref{ap5b}. Therefore, the function $F_\xi$ is necessarily identically zero by the 
principle of isolated zeros, and we have
\bel{ap10}
A_n(\xi)=0,\ \xi \in \R,\ n \in \N,
\ee
according to \eqref{ap8b}.

\noindent {\it Step 5: Computation of $A_2$.}
With reference to \eqref{ap8c} we start by decomposing $\omega(\cdot,\xi) r_\xi(z) \phi_1(\cdot,\xi)$, $\xi \in \R$, on the $L^2(\R)$-orthonormal basis $\{ \phi_j(\cdot,\xi),\ j \geq 1 \}$. We get that
$$ \omega(\cdot,\xi) r_\xi(z) \phi_1(\cdot,\xi) = \frac{\omega(\cdot,\xi)}{\lambda_1(\xi)-z} \phi_1(\cdot,\xi)= \frac{1}{\lambda_1(\xi)-z}\sum_{j=1}^\infty \omega_j(\xi) \phi_j(\cdot,\xi), $$
where $\omega_j(\xi):=\langle \omega(\cdot,\xi) \phi_1(\cdot,\xi) , \phi_j(\cdot,\xi) \rangle_2$. Therefore, we have
$$ \langle (\omega(\cdot,\xi) r_\xi(z))^2 \phi_1(\cdot,\xi) , \phi_1(\cdot,\xi) \rangle_2
=\frac{1}{\lambda_1(\xi)-z} \sum_{j =1}^\infty \frac{\abs{\omega_j(\xi)}^2}{\lambda_j(\xi)-z}, $$
and \eqref{ap8c} then yields that
\begin{align}
A_2(\xi) & =   \| w \phi_1(\cdot,\xi) \|_2^2 + \sum_{j =1}^\infty \frac{\abs{\omega_j(\xi)}^2}{2 i \pi} \int_{C(\lambda_1(\xi),\rho)} \frac{\dd z}{(\lambda_1(\xi)-z)(\lambda_j(\xi)-z)} \label{ap11}  \\
& =    \| w \phi_1(\cdot,\xi) \|_2^2 -\sum_{j =2}^\infty \frac{\abs{\omega_j(\xi)}^2}{\lambda_j(\xi)-\lambda_1(\xi)},\ \xi \in \R, \nonumber
\end{align}
by straightforward computation.
The next step is to make the second term on the right-hand side of \eqref{ap11} sufficiently small relative to $\| w \phi_1(\cdot,\xi) \|_2^2$ by choosing $\xi$ suitably in $\R$.

\noindent{\it Step 6: End of the proof.} 
We refer to \eqref{m1} and notice that there exists $\kappa \in (0,1)$ such that
$$ 
r <r(\kappa) := (1-\kappa)^{1 \slash2}\frac{(3b_- - b_+)^{1 \slash2}}{2b_+}.
$$
Further, since $\abs{a(x)} \le b_+ \abs{x}$ for all $x \in \R$, by \eqref{a2b}-\eqref{a3}, we get that
$$ \abs{v(x,\xi)} \le 2(\abs{\xi} + b_+ r),\ x \in K,\ \xi \in \R, $$
and hence that
\bel{ap13}
\| v(\cdot,\xi) \|_{L^\infty(K)} \le (1 -\kappa)^{1 \slash 2} (3b_- - b_+)^{1 \slash2},\ \xi \in  [-\xi(\kappa),\xi(\kappa)],
\ee
where $\xi(\kappa):=b_+ (r(\kappa) - r)$.
Moreover, we see from \eqref{a5} that
$$\lambda_j(\xi)-\lambda_1(\xi) \ge (2j-1) b_- - b_+ \ge 3b_- - b_+,\ \xi \in \R,\ j \ge 2, $$
and consequently 
$$
\sum_{j=2}^\infty \frac{\abs{\omega_j(\xi)}^2}{\lambda_j(\xi)-\lambda_1(\xi)} \le \sum_{j=2}^\infty \frac{\abs{\omega_j(\xi)}^2}{3b_- - b_+} 
\le \frac{\| \omega(\cdot,\xi) \phi_1(\cdot,\xi) \|_2^2}{3b_- - b_+},
$$
by the Plancherel formula.
It follows from this and the identity $\omega(\cdot,\xi)=-v(\cdot,\xi) w$ for all $\xi \in \R$, that 
$$ \sum_{j= 2}^\infty \frac{\abs{\omega_j(\xi)}^2}{\lambda_j(\xi)-\lambda_1(\xi)} \le \frac{\| v(\cdot,\xi) \|_{L^\infty(K)}^2}{3 b_- - b_+} \| w \phi_1(\cdot,\xi) \|_2^2,\ \xi \in \R. $$
Thus, taking $\xi \in [-\xi(\kappa),\xi(\kappa)]$ in the above line, we deduce from \eqref{ap13} that
$$ 
\sum_{j=2}^\infty \frac{\abs{\omega_j(\xi)}^2}{\lambda_j(\xi)-\lambda_1(\xi)} \le (1-\kappa) \| w \phi_1(\cdot,\xi) \|_2^2. $$
Therefore, we get from \eqref{ap11} that
$$ A_2(\xi) \ge \kappa \| w \phi_1(\cdot,\xi) \|_2^2,\ \xi \in [-\xi(\kappa),\xi(\kappa)],$$
and then from \eqref{ap10} that $w \phi_1(\cdot,\xi)=0$ in $L^2(\R)$.  Finally, 
taking $\xi=0$ (which is permitted since it lies in the interval $[-\xi(\kappa),\xi(\kappa)]$) and bearing in mind that $\phi_1(x,0) >0$ for a.e. $x \in \R$,  we obtain that $w=0$ a.e. in $\R$ and hence the desired result.

\subsection{Proof of Theorem \ref{thm2}}
Let us first notice that the ground state determines the potential up to an additive constant. More precisely, we have
$$ q(x,\xi_0)=\lambda_1(\xi_0)+\frac{\varphi_1''(x,\xi_0)}{\varphi_1(x,\xi_0)},\ x \in \R, $$
from \eqref{r1} and \eqref{qv1}, 
because $\varphi_1(\xi_0)$ nowhere vanishes according to \eqref{qv2}. 

In the context of Theorem \ref{thm2}, we have $\varphi_1(\cdot,\xi_0)=\tilde{\varphi}_1(\cdot,\xi_0)$, which entails that
$$ (\xi_0-a(x))^2-(\xi_0-\tilde{a}(x))^2=C,\ x \in \R, $$
where $C:=-(\lambda_1(\xi_0)-\tilde{\lambda}_1(\xi_0))$ is seen to vanish by setting $x=0$. Hence
$$  \xi_0-a(x)=\pm(\xi_0-\tilde{a}(x)),\ x \in \R, $$
where the sign $\pm$ could in principle depend on $x$. We show momentarily that it is $+$ throughout, thus implying $a=\tilde{a}$ and hence the conclusion.

For suppose that the sign is $-$ at $x_0 \in \R$ and that it is so without choice. Then $\xi_0-a(x_0) \neq 0$ and hence the sign of
$\xi_0-a(x)<0$ in a neighborhood of $x_0$. But that contradicts the fact that $a$ and $\tilde{a}$ are increasing.

\end{document}